\documentclass[11pt,a4paper,notitlepage]{article}
\usepackage[utf8]{inputenc}
\usepackage{amsmath}
\usepackage{amsfonts,dsfont}
\usepackage{amssymb}
\usepackage{amsthm}
\usepackage{graphicx}
\usepackage{pstricks}

\usepackage{helvet}
\input{amssym.def} 
\input{amssym}
\usepackage{enumitem, hyperref}
\makeatletter
\def\namedlabel#1#2{\begingroup
    #2%
    \def\@currentlabel{#2}%
    \phantomsection\label{#1}\endgroup
}
\makeatother

\newtheorem{defi}{Definition}[section]
\newtheorem{rem}{Remark}[section]
\newtheorem{thm}{Theorem}[section]
\newtheorem{lem}[thm]{Lemma}
\newtheorem{prop}[thm]{Proposition}

\newcommand{\norm}[1]{\Vert #1 \Vert}
\newcommand{\abs}[1]{\left| #1 \right|}

\renewcommand{\div}{\mathrm{div}\,}
\newcommand{\vect}[1]{\mathrm{span}( #1 )}
\newcommand{\ie}{\emph{i.e. }}

\newcommand{\cf}{\emph{cf }}
\def\T #1{\mathcal{T} #1 }
\def\L #1{\mathrm{L}^{ #1 }}
\def\H #1{\mathrm{H}^{ #1 }}
\def\W #1{\mathrm{W}^{ #1 }}
\def\V{\mathrm{V}}
\def\d{\mathrm{d}}

\title{Existence theorem for homogeneous incompressible Navier-Stokes equation with variable rheology}
\author{Laurent Chupin\thanks{Universit\'e Blaise Pascal, Clermont-Ferrand II, Laboratoire de Math\'ematiques CNRS-UMR 6620, Campus des C\'ezeaux, F-63177 Aubi\`ere cedex, France ({\tt laurent.chupin@math.univ-bpclermont.fr}).}
        \and \& \and Jordane Math\'e\thanks{Universit\'e Blaise Pascal, Clermont-Ferrand II, Laboratoire de Math\'ematiques CNRS-UMR 6620 and Laboratoire Magmas et Volcans (LMV) ({\tt jordane.mathe@math.univ-bpclermont.fr}).} }

\begin{document}
\maketitle

\begin{abstract}
We look at a homogeneous incompressible fluid with a time and space~variable rheology of Bingham type, which is governed by a coupling equation. Such a system is a simplified model for a gas-particle mixture that flows under the influence of gravity. The main application of this kind of model is pyroclastic flows in the context of volcanology. In order to prove long time existence of weak solutions, classical Galerkin approximation method coupled with \emph{a priori} estimates allows us to partially solve the problem. A difficulty remains with the stress tensor, which must satisfy an implicit constitutive relation. Some numerical simulations of a flow of this type are given in the last section. These numerical experiments highlight the influence of the fluidization phenomenon in the flow.
\end{abstract}

\smallskip


\section{Introduction}
Dense pyroclastic flows are fast-moving density currents emitted by volcanoes that are particularly destructive. 
They are composed of high temperature crushed rocks ranging in size from micrometers to meters. 
They are granular flows but they show a very high fluidity: flows thinner than one meter can reach run-out distances longer than 10 kilometers and they are able to flow on horizontal surfaces. 
One objective of modern volcanology is to better understand the behavior of this type of flows for a better assessment of volcanic hazards. 
A possible explanation of the high mobility is the interaction between the fine particles of the flow and volcanic or atmospheric gases. 
Some physical experiments have been made on the collapse of a fluidized dense granular column in a U-shaped channel, see for instance~\cite{I1,PR} or~\cite{R}.
In these experiments, the granular bed is initially fluidized by injection of air through the column, which generates high interstitial pore fluid pressure so that the released granular material flows like an inertial single-phase fluid.
In this context, a lot of mathematical models have been proposed to take into account the change of rheology all along the flow: \cite{AT,ATMC,DP,FP,I2,MTVAB,CM,VTA}.
A first category concerns bi-fluids systems, in which each phase is described by a Navier-Stokes type equation added to a coupling term (see~\cite{CM}). 
An important problem with this kind of model is their heaviness in terms of numerical computations.
This is why plenty of models are of Shallow-Water type, provided the length of the flow is much bigger than the thickness, see~\cite{I2}.
This method allows to reduce a $d$-dimensional problem to a $(d\!-\!1)$-dimensional one, and is then much more efficient for computational experiments. But the cost is the loss of information in the direction which seems to be crucial to understand the behavior of fluidized granular flows.
In order to overcome the problems above it is possible to consider a homogeneous incompressible fluid with a rheology depending on a parameter (called \emph{order parameter} in \cite{ATMC,MTVAB,VTA}), which is supposed to satisfy an equation describing its evolution.
For such a model, the equations governing the hydrodynamics are the Navier-Stokes equations including an extra-stress tensor. 
They read
\begin{equation*}
\left\{
\begin{aligned}
& \rho \, \big( \partial_t v + v \cdot \nabla  v \big) + \nabla p - \eta \, \Delta v = \div \sigma + \rho \, g, \\
& \div v = 0,
\end{aligned}
\right.
\end{equation*}
where~$v$ is the velocity, $p$ the pressure and~$\sigma$ is the extra-stress tensor. The vector~$g$ characterizes the gravity forces, the coefficient~$\eta>0$ refers to the Newtonian viscosity and~$\rho > 0$ is the constant density of the mixture.
Note that the constant density of the mixture can be expressed as the linear combination of the constant density of each component of the mixture:
$$
\rho = \pi \rho_s + (1-\pi) \rho_f ,
$$ 
where~$\rho_f$ is the constant density of the gas,~$\rho_s$ the constant density of the solid particles and~$\pi$ is the mass fraction of gas in the mixture.
In order to take into account the behaviour of dense granular flows, the extra-stress~$\sigma$ depends not only on the symmetric part of the velocity gradient~$Dv$ but also on an order-parameter~$q$ evolving with the flow (\cf \cite{perso}):
$$
\sigma = \sigma(q, Dv) \qquad \text{with} \qquad Dv = \frac{\nabla v + (\nabla v )^T}{2}.
$$
More precisely, the function~$\sigma$ is defined using a Bingham type relation:
\begin{equation}\label{bingham}
 \left\{
\begin{aligned}
& \sigma(q, Dv) = q \, \frac{Dv}{\abs{Dv}} && \quad \text{ if } Dv \neq 0 ,\\
& \abs{\sigma(q, Dv) }\leq q            && \quad \text{ if } Dv = 0,
\end{aligned}
\right.
\end{equation}
where for any tensor~$A$, $\abs{A} = \sqrt{A:A} = \sqrt{\sum_{i,j}A_{i,j}^2}$.
In other words, we want to impose a solid behaviour when the stress is smaller than a plasticity yield, named~$q$.
In the model presented in this paper, contrary to the ones in~\cite{ATMC,MTVAB,VTA}, the order-parameter depends on the interstitial pore fluid pressure, denoted by~$p_f$.
It means that in this model the gas present between solid particles only acts on the global dynamic \emph{via} its local pressure.
More precisely, our order parameter is in fact the plasticity yield defined as a function of the difference between~$p_f$ and a solid pressure~$p_s$ as follows:
\begin{equation}
\label{def q}
q = q_0 (p_s - p_f)^+, \qquad q_0>0.
\end{equation}
This choice means that when the fluid pressure compensates the solid pressure then the mixture is Newtonian ($q=0$), whereas when fluid pressure is lower then plasticity effects happen ($q>0$).
The model must be completed by the relation giving such pressures~$p_f$ and~$p_s$. 

$\bullet$ Let's begin with the pore fluid pressure~$p_f$.
As in~\cite{R1}, some experiments show that, in the static case, the evolution of this pressure is given by a diffusion law.
To understand what happens in a flowing case, one can follow Iverson and Denlinger's inference, which gives an advection-diffusion law for the behaviour of the basal pore pressure in a shallow-water type model, see~\cite{I2}.
Same phenomenological observations allow us to say that the fluid pressure~$p_f$ verifies, in our model, an advection-diffusion equation like:
$$ \partial_t p_f + v\cdot \nabla p_f - K \, \Delta p_f = 0 ,$$
\indent $\bullet$ We assume that the solid pressure~$p_s$ only depends on the gravity.
More precisely, it corresponds to the pressure that imposes a column of grain particles anywhere in the bed.
Then it can be explicitly defined as the lithostatic solid pressure, \ie 
\begin{equation}\label{def_ps}
p_s = {p_s}_0 \, + \, \rho \, |g| \, (y_0-y), 
\end{equation}
where~${p_s}_0 \geq 0$ is a reference pressure given at the altitude~$y = y_0$ ($y$ begin the vertical coordinate).

\begin{rem}
Note that there are three pressures in the model: $p_s$ is the lithostatic (or hydrostatic) pressure, $p_f$ is the pressure of the gas which fills the gap between solid particles, and $p$ is the hydrodynamic pressure of the mixture.
As a result, in the particular case of a static state of the solid matter, $p$ equals $p_s$.
In~\cite{R1}, this solid pressure represents the minimum fluidization pressure measured on the base of a static grain column during its fluidization.
More precisely, this is the pressure of the gas once the velocity of injection is sufficient to impose an upward drag force which counterbalances the weight of particles.
But there is no \emph{a priori} relation between $p_s$, $p_f$ and $p$ in the general case.
\end{rem}

Finally, the dynamic of such a fluid is governed by a system of three equations with three unknowns:~$v$, $p$ and $p_f$, which is
\begin{equation}\label{modelsimpli0}
\left\{
\begin{aligned}
& \rho \, \big( \partial_t v + v \cdot \nabla v \big) + \nabla p - \eta \, \Delta v = q_0 \, \div \Big( (p_s - p_f)^+ \frac{Dv}{\abs{Dv}} \Big) + f, \\
& \div v = 0, \\[0.2cm]
& \partial_t p_f + v \cdot \nabla p_f - K \, \Delta p_f = 0,
\end{aligned}
\right.
\end{equation}
where~$p_s$ is the function given by~\eqref{def_ps} and~$f$ is a source term taking into account all exterior forces, for instance the gravity forces when~$f = \rho g$. 
In the case where the solid pressure and fluid pressure are equal ($p_s=p_f$) then we obtain a pure viscous flow (modeled by the incompressible Navier-Stokes equations).
If the difference between these two pressures is assumed to be constant then we obtain the classical Bingham model.
These two regimes are already well known, see for example~\cite{B-F}. 
In order to solve the existence problem for such complex fluids, a general theory has been developed and studied in~\cite{B-G-M-S}, in which the authors prove the existence of a weak solution for a general class of rheology including fluids of a Bingham and Herschel-Bulkley type and depending on both time and space. 
But this framework does not include our case based on coupling equations.\\

\noindent
In this article, we generalize these results to our physical model by treating the coupling relation through the definition of the order-parameter.
Even if the result given in~\cite{B-G-M-S} is true in the three-dimensional case, according to Remark~\ref{3D}, we will restrict the study to the two-dimensional case in this article.\\

\noindent
This paper is organized as follows. 
In Section~\ref{section1}, we introduce the mathematical formulation of the problem and give the main result: the existence of a weak solution.
The Section~\ref{section2} is devoted to the proof of this result: a construction of approximated solutions is given in the Subsection~\ref{section2-1}. 
We next prove uniform estimates in the Subsection~\ref{section2-2}, and we finally deduce the existence result using a limit process in the Subsection~\ref{section2-3}.
In the last section, we give some numerical results on the classical Poiseuille-type flow in a 2D-channel, that show the impact of the fluidization on the behaviour of a Bingham fluid.

\section{Mathematical framework and main result}\label{section1}

In order to describe the notion of solution to the system~\eqref{modelsimpli0}, we need to introduce suitable functional spaces.

\subsection{Functional spaces}\label{Functional spaces}

In the sequel, the fluid domain $\Omega\subset \mathbb R^2$ is a bounded connected open set with a smooth boundary denoted by~$\partial \Omega$.
Its outward unit normal vector is denoted by~$n$.

$\checkmark$ According to classical notations, we will denote by~$\L{m}(\Omega)$, $m \geq 1$, Lebesgue spaces on~$\Omega$ and~$\H{s}(\Omega)$, or more generally~$\W{s,q}(\Omega)$ ($s>0$ and $q \geq 1$), Sobolev spaces on~$\Omega$.

$\checkmark$ To deal with boundary conditions and incompressibility constraint, we consider the following notation
\begin{equation*}
\begin{aligned}
& \H{1}_0(\Omega) = \{ g \in \big(\H{1}(\Omega)\big)^2 ~;~ g\vert_{\partial\Omega} = 0 \}, \\
& \H{} = \{ g \in \big(\L{2}(\Omega)\big)^2 ~;~ \div g = 0 ~ \text{and} ~ g\cdot n = 0 ~ \text{on $\partial \Omega$} \},\\
& \V = \{ g \in \big(\H{1}_0(\Omega)\big)^2 ~;~ \div g = 0 \}.
\end{aligned}
\end{equation*}
As usual, the dual space of~$\H{1}_0(\Omega)$ and~$V$ will be denoted by~$\H{-1}$ and~$V'$ respectively.

$\checkmark$ $\mathcal{D}(\Omega) = \mathcal{C}^\infty_c (\Omega)$ will be the notation for the set of infinitely differentiable functions with compact support on $ \Omega $.

$\checkmark$ For time-dependent functions, we denote by~$\L{m}X$ the set also denoted by~$ \L{m}(0,T;X)$, $0 < T < + \infty$ and $m \geq 1$, defined by 
$$\L{m}X = \{ g(t,\cdot) \in X \text{ for a. a. } t\in(0,T)  ~;~ t \mapsto \norm{g(t,\cdot)}_X \in \L{m}(0,T) \},$$
with~$\norm{\cdot}_X$ being a norm on~$X$, where~$X$ is any separable functional space as above. \\
\noindent In the same way, we denote by~$\mathcal C ^0 ([0,T],X)$ the set of time-continuous functions on~$[0,T]$ to the space~$X$.

$\checkmark$ For a given basis~$(x_i)_i$ of any functional space~$X$ from above, we will denote by~$X_k$ the finite dimensional subspace of~$X$ of the form
$$ X_k = \vect{x_1, \dots, x_k}.$$

$\checkmark$ Finally, because lots of tensors used in this paper are symmetric, we will denote this with a subscript
$$
\L{m}_{sym} = \{ A \in {\L{m}(\Omega)}^{2 \times 2} ~;~ A \text{ is symmetric}\},
$$
and the set of $2 \times 2$ real symmetric matrices will be denoted by~$S_2(\mathbb{R})$.

\subsection{Rigorous formulation of the System~\eqref{modelsimpli0}}

In System~\eqref{modelsimpli0}, one element is, at this stage, not completely well defined. 
Indeed, the fraction~$\frac{Dv}{|Dv|}$ makes no sense on the (rigid) set~$\left\{ Dv=0 \right\}$.
Therefore, instead of using definition~\eqref{bingham} for $\sigma(q,Dv)$, it is preferred to use formulation below:
\begin{equation}\label{bingham1}
\left\{
\begin{aligned}
& \sigma : Dv = q \abs{Dv} ,\\
& \abs{\sigma} \leq q.
\end{aligned}
\right.
\end{equation}

\begin{lem}
Both systems~\eqref{bingham} and~\eqref{bingham1} are equivalent, 
 more precisely for all $(\sigma, Dv) \in \big(\mathbb R^{2\times 2}\big)^2$ and for all $ q \geq 0 $, we have
$$
 \left\{
\begin{aligned}
& \sigma = q \, \frac{Dv}{\abs{Dv}} &&\text{if } Dv \neq 0\\
& \abs{\sigma}\leq q                &&\text{if } Dv = 0
\end{aligned}
\right.
~~ \Longleftrightarrow ~~
\left\{
\begin{aligned}
& \sigma : Dv = q \abs{Dv}\\
& \abs{\sigma} \leq q .
\end{aligned}
\right.
$$
\end{lem}
\begin{proof}
Let's consider first the case with~$q = 0$. 
In this case, both systems are reduced to $\sigma = 0$ so that they are equivalent.\\
Let's suppose now that $q > 0$.
It is obvious that the system on the left~\eqref{bingham} implies the one on the right \eqref{bingham1}.
To prove the converse, we assume that~\eqref{bingham1} is true and we split the proof into two cases:
\begin{itemize}
	\item Case 1: $Dv = 0$.\\
	      If $ Dv = 0$, then the second relation of the system~\eqref{bingham} is obviously true.
	\item Case 2: $Dv \neq 0$.\\
      	Else, taking the absolute value in the first relation of~\eqref{bingham1}, we compute the scalar equality
				$$\abs{\sigma : Dv} = q \abs{Dv} .$$
				Whereas applying the Cauchy-Schwarz inequality and using the second relation of~\eqref{bingham1}, we get 
				$$ \abs{\sigma : Dv} \leq \abs{\sigma} \abs{Dv} \leq q \abs{Dv} .$$
				This inequality, together with the previous equality, gives $ q \abs{Dv} \leq \abs{\sigma} \abs{Dv} \leq q \abs{Dv}$.
				Then remember that $Dv \neq 0$, so that we have $ \abs{\sigma} = q$ and thus
				$$
				\abs{\sigma : Dv} = \abs{\sigma} \abs{Dv} .
				$$
				In this case, using a corollary of Cauchy-Schwarz inequality, we also know that the tensors $\sigma$ and $ Dv $ are linearly dependent.
				That means there is a scalar~$a$ such that $\sigma = a Dv$.
				Using this expression together with the first line of the system~\eqref{bingham1}, we conclude that this scalar exactly corresponds to~$a = \frac{q}{\abs{Dv}}$, which concludes the proof.
\end{itemize}
\end{proof}

\noindent
Thereby, the problem is formulated as follows.
Being given initial data velocity~$v_0$, initial fluid pressure~${p_f}_0$, a source term~$f$ and a time $T>0$, we look for a velocity vector field~$v$, two scalar pressures~$p$ and~$p_f$, and a symmetric stress tensor~$\sigma$ such that
\begin{equation}\label{modelsimpli}
\left\{
\begin{aligned}
& \rho \big( \partial_t v + v \cdot \nabla v \big) + \nabla p - \eta \, \Delta v = \div \sigma + f && \text{in} ~ \Omega \times (0,T), \\
& \div v = 0 && \text{in} ~ \Omega \times (0,T), \\
& \partial_t p_f + v \cdot \nabla p_f - K \, \Delta p_f = 0 && \text{in} ~ \Omega \times (0,T), \\
& q=q_0 \, (p_s - p_f)^+ ~\text{where $p_s$ is defined by~\eqref{def_ps}} \qquad && \text{in} ~ \Omega \times (0,T), \\
& (\sigma, Dv, q) ~\text{satisfies~\eqref{bingham1}} && \text{in} ~ \Omega \times (0,T), \\
& (v, \, p_f)\vert_{t=0} = (v_0, \, {p_f}_0) && \text{in} ~ \Omega, \\
& (v, \, p_f) = (0, \, 0) && \text{in} ~ \partial \Omega \times (0,T).
\end{aligned}
\right.
\end{equation}
We would like to point out that, as an assumption, we restrict the boundary conditions for the interstitial pressure~$p_f$ and the velocity~$v$ to homogeneous Dirichlet conditions.
Note that this restriction is essentially due to mathematical convenience when working with the variational formulation and that adaptation to a more general case will be not treated in this paper.

\subsection{Definition of a weak solution and existence theorem}

In order to work with non-regular solutions, and then to show global existence in time result, we use a weak formulation of the system~\eqref{modelsimpli}.
The weak interpretation of the Navier-Stokes equation being standard, we only detail the weak formulation of the condition~\eqref{bingham1}.
For any bounded function~$q$, we introduce the following graph:
\begin{equation}\label{definition-Gq}
\mathcal{G}[q] = \left \lbrace (\sigma,Dv) \in \L{2}\L{2}_{sym} \times (\L{\infty}\H{-1}\cap\L{2}\L{2}) \,;\, \text{$(\sigma, Dv, q)$ satisfies~\eqref{bingham1} a.e.}\right\rbrace,
\end{equation}
and we say that $(\sigma, Dv, q)$ satisfies~\eqref{bingham1} in a weak sense if $(\sigma,Dv)\in \mathcal{G}[q]$. 
One of the most important properties of the graph~$\mathcal{G}[q]$ is its monotony. 
More exactly, we have the following result which has been first introduced in this context in~\cite{B-G-M-S} and is fundamental at the end of the proof (see Section~\ref{limit-graph}-Step~2~(b).3) to reach some limits in the graph~$\mathcal{G}[q]$. 
\begin{lem}\label{graphe positif} 
Let~$q$ be a given bounded function.
If $(\sigma_1,D_1)$ and $(\sigma_2,D_2)$ are in~$\mathcal{G}[q]$, then 
$$ (\sigma_1 - \sigma_2):(D_1 - D_2) \geq 0 \quad \text{ a.e.}$$
\end{lem}
\begin{proof}
Suppose that $(\sigma_1,D_1,q)$ and $(\sigma_2,D_2,q)$ satisfy~\eqref{bingham1}. 
We have
$$
\begin{aligned}
(\sigma_1 - \sigma_2):(D_1 - D_2) 
& = \sigma_1 : D_1 + \sigma_2 : D_2 - \sigma_1 : D_2 - \sigma_2 : D_1 \\
& = q \abs{D_1} + q \abs{D_2} - \sigma_1 : D_2 - \sigma_2 : D_1.
\end{aligned}
$$
Using the Cauchy-Schwarz inequality we deduce
$$
(\sigma_1 - \sigma_2):(D_1 - D_2) 
\geq
q \abs{D_1} + q \abs{D_2} - \abs{\sigma_1}\abs{D_2} - \abs{\sigma_2}\abs{D_1} \geq 0,
$$
which concludes the proof.
\end{proof}

\noindent

\begin{defi}\label{definition-weak}
Let $f \in \L{2}(\mathbb{R}_+;V')$ and $(v_0, {p_f}_0) \in \H{} \times \L{\infty}$ be the initial conditions.
For $T > 0$, we say that $(v, p_f, \sigma)$ is a \emph{weak solution} of~\eqref{modelsimpli} on $\Omega \times (0,T)$ if
\begin{equation*}
\begin{aligned}
& v\in \L{\infty}\mathrm{H} \cap \L{2}\mathrm{V}, && \qquad \partial_t v \in \L{2}\V', \\
& p_f \in \L{\infty}\L{\infty} \cap \L{2}\H{1}_0, && \qquad \partial_t p_f \in \L{2}\H{-1}, \\
& \sigma \in \L{2}\L{2}_{sym},
\end{aligned}
\end{equation*}
and satisfy, for any $(\psi, r) \in \mathrm{V} \times \H{1}_0(\Omega)$:
\begin{equation}\label{formulation faible1}
\left\{
\begin{aligned}
& \rho \bigg( \left\langle \partial_t v, \psi \right\rangle_{\V',\V} + \int_\Omega (v \cdot \nabla v) \cdot \psi \bigg) + \eta \int_\Omega \nabla v : \nabla \psi + \int_\Omega \sigma : \nabla \psi = \left\langle f, \psi \right\rangle_{\V',\V}, \\
& \left\langle \partial_t p_f , r \right\rangle_{\H{-1},\H{1}_0} + \int_\Omega \left( v \cdot \nabla p_f \right) \, r + K \int_\Omega \nabla p_f \cdot \nabla r = 0, \\
& (\sigma,Dv) \in \mathcal{G}[q] \quad \text{with} ~~ q = q_0 (p_s - p_f)^+, \\[0.2cm]
& (v, p_f) \vert _{t=0} = (v_0, {p_f}_0) \text{ a.e. in } \Omega,
 \end{aligned}
 \right.
\end{equation}
where~$p_s = {p_s}_0 \, + \, \rho \, |g| \, (y_0-y)$ is a given function of the density and the vertical coordinate~$y$.
\end{defi}

\noindent
The main result of the paper is the following theorem.
%
%
\begin{thm}\label{theorem}
If $f \in \L{2}(\mathbb{R}_+;V')$, $v_0 \in \H{}$ and
~${p_f}_0 \in \L{\infty}(\Omega)$ is non-negative,
then for all $T > 0$, there exists a weak solution $(v,p_f,\sigma)$ of~\eqref{modelsimpli} on $\Omega \times (0,T)$ in the sense of the definition~\ref{definition-weak}, with~$p_f \geq 0$ a.e. in $(0,T)\times\Omega$.
\end{thm}

\section{Proof}\label{section2}

The proof of the existence Theorem~\ref{theorem} consists of two main steps:
\begin{enumerate}
\item We first introduce an approximate solution using a sequence of well-posed problems. 
The first difficulty is common: 
we look for a solution $(\sigma,p_f,v)$ with values in infinite dimensional spaces (of kind~$\L{2}$) and we use a Galerkin scheme to approach it by functions with values in finite dimensional spaces (for instance~$\L{2}_k$), and then use the classical Cauchy theorem for ODE.
The second difficulty comes from the Bingham condition $(\sigma,Dv)\in \mathcal{G}[q]$. Even in the finite dimensional case, this kind of condition is not usual and we introduce a regularization: we approach the graph~$\mathcal{G}[q]$ by the graph of a regular function $Dv\longmapsto \sigma$;
\item Next, we take the limit in the approximate solution to find a solution to the initial problem. 
We then use the energy estimates of the model to deduce bounds. 
The main difficulty is to obtain the limit in the non-linear condition $(\sigma,Dv)\in \mathcal{G}[q]$. 
This point is treated in the subsection~\ref{limit-graph}.
\end{enumerate} 

\subsection{Approximate model}\label{section2-1}

Let $k > 0$ be an integer.
\paragraph{Galerkin approximation}
In a classical way, we reduce the problem to finite dimensional spaces following Galerkin approximations. 
Let $(w_i)_{i\geq 1}$ be the basis of~$\V$ made of eigenfunctions of Stokes' operator. 
Let $(r_i)_{i\geq 1}$ be the basis of~$\H{1}_0(\Omega)$ made of eigenfunctions of the operator~$- \Delta$. 
Let~$\mathcal{P}_k$ be the orthogonal projection from~$\V$ to~$ \V _k$, finite dimensional space spanned by the~$k$ first eigenfunctions of Stokes' operator~$\{ w_1, \dots, w_{k} \}$. 
In the same way,~$\mathcal{Q}_k$ is the orthogonal projection from~$\H{1}_{0}$ to~$\H{1}_{0,k} := \mathrm{span}(r_1,\dots, r_{k})$.
We define the projected velocity and the projected fluid pressure by
$$
v_k = \mathcal{P}_k (v)
\quad \text{and} \quad
p_{f,k} = \mathcal{Q}_k(p_f).
$$
\paragraph{Graph regularization}
For any given $q_k$, we define a regularized graph $\mathcal G _k$ by
\begin{equation}
\mathcal{G}_k = \Big\{ \Big( q_k \frac{Dv}{\abs{Dv}+1/k},Dv \Big) \, , \; v \in\L{2}\V_k \Big\}.
\label{eq:regular graph}
\end{equation}
\paragraph{Regularization of the initial conditions}
The initial velocity~$v_0$ is approached using the projector above:
$$
v_0^k = \mathcal{P}_k (v_0).
$$
Since the initial fluid pressure~${p_f}_0$ is non-negative (and since the fluid pressure satisfies the maximum principle), in order to conserve this property at each step of the approximation, we define
$$
{p_f}_0^k = \T{\big( \mathcal{Q}_k({p_f}_0) \big) },
$$
where~$\T{}$ is the truncation between~$0$ and~$ \operatorname*{ess\,sup}_{\Omega} {p_f}_0$. 
\paragraph{Approximated problem}
Now, the approximated problem is the following:
\par
\noindent
Find $(v_k,p_{f,k},\sigma_k) \in ( \L{\infty}\mathrm{H} \cap \L{2}\mathrm{V}_k )  \times ( \L{\infty}\L{\infty} \cap \L{2}\H{1}_{0,k} ) \times \L{2}\L{2}_{sym} $ such that, for all $(\psi,r)\in \mathrm{V}_k \times \H{1}_{0,k}$ we have
\begin{equation}\label{model k}
\left\{
\begin{aligned}
& \rho \bigg( \left\langle \partial_t v_k, \psi \right\rangle_{\V',\V} + \int_\Omega ( v_k \cdot \nabla v_k ) \cdot \psi \bigg) + \eta \int_\Omega \nabla v_k : \nabla \psi + \int_\Omega \sigma_k : \nabla \psi = \left\langle f, \psi \right\rangle_{\V',\V}, \\
& \left\langle \partial_t p_{f,k} , r \right\rangle_{\H{-1},\H{1}_0} + \int_\Omega (v_k \cdot \nabla p_{f,k}) \, r + K \int_\Omega \nabla p_{f,k} \cdot \nabla r = 0, \\
& (\sigma_k,Dv_k) \in \mathcal{G}_k \text{ with } q_k = q_0(p_s - p_{f,k} )^+ ,\\[0.2cm]
& (v_k, \, p_{f,k}) \vert _{t=0} = (v_0^k, \, {p_f}_0^k).
\end{aligned}
\right.
\end{equation}
Thanks to the definition of the graph~$\mathcal{G}_k$, see the equation~\eqref{eq:regular graph}, the stress tensor~$\sigma_k$ is explicitly given as a function of~$q_k$ and~$v_k$, so we have decoupled the problem.
We can then solve problem~\eqref{model k} using the Schauder theorem and classical estimates (see~\cite{B-F}). 

\subsection{A priori estimates}\label{section2-2}

The aim of this subsection is to obtain suitable estimates on the solution $(v_k,p_{f,k},\sigma_k)$ to the approximated problem~$(\ref{model k})$, which do not depend on the parameter~$k$. 

\begin{prop}\label{prop:estim k 1}
The solution $(v_k,p_{f,k},\sigma_k)$ to the problem~$(\ref{model k})$ satisfies the following estimates (where $c$ is a constant which does not depend on~$k$):
\begin{enumerate}
\item Uniform bounds on the fluid pressure, plasticity yield and stress tensor
\begin{equation}\label{estimate1}
0 \leq p_{f,k} \leq \operatorname*{ess\,sup} {p_f}_0, \quad \norm{q_{k}}_{\L \infty \L\infty} \leq c, \quad \norm{\sigma_{k}}_{\L \infty \L\infty} \leq c.
\end{equation}
\item Energy estimate on the velocity field
\begin{equation}\label{estimate2}
\norm{v_k}_{\L{\infty}\L{2}}^2 + \norm{\nabla v_k}_{\L{2}\L{2}}^2 + \norm{\sigma_k:Dv_k}_{\L{1}\L{1}} \leq c.
\end{equation}
\item Estimate on the momentum
\begin{equation}\label{estimate4}
\norm{\partial_t v_k}_{\L{2}\V'} \leq c.
\end{equation}
\item Estimate on the fluid pressure and its time derivative
\begin{equation}\label{estimate3}
\norm{p_{f,k}}_{\L{2}\H{1}} \leq c \quad \text{and} \quad \norm{\partial_t p_{f,k}}_{\L{2}\H{-1}} \leq c.
\end{equation}
\end{enumerate}
\end{prop}

\begin{rem}
In the sequel, we will want to use approximated solutions~$v_k$ and~${p_f}_k$ as test-functions in~\eqref{model k}. 
It is theoretically not possible because of the time dependency of these functions. 
In fact, it is possible to avoid the problem, remarking that these functions are linear combinations of space-dependent eigenfunctions, 
with time-dependent coefficients. For example, function~$v_k$ reads as
$$ v_k(x,t) = \sum_{i=1}^{k} c_i(t) w_i(x) .$$
Therefore, it suffices to choose the eigenfunctions as test-functions, multiply each equation by the appropriate coefficient and sum up the results to get the same equation we would have if we formally used $v_k$ as a test-function, see~\cite{B-F}, p.349.
\end{rem}

\begin{proof}
Consider a solution $(v_k,p_{f,k},\sigma_k)$ to the approximated problem~\eqref{model k}.
\begin{enumerate}
\item The first estimate of~\eqref{estimate1} comes from to the maximum principle. We also use the definition of the truncation~$\mathcal T$ and the bounds on the initial condition ${p_f}_0$ in $\L{\infty}$. 
The second estimate of~\eqref{estimate1} is a consequence of the first estimate since we recall that~$q_k  = q_0(p_s- p_{f,k})^+$.
The last estimate of~\eqref{estimate1} comes from the definition of $\sigma_k$.
\item In order to obtain a so-called energy estimate for the velocity we use the usual method: we chose $\psi=2v_k$ as test-function in the weak formulation~\eqref{model k}. 
After classical calculation we obtain
\begin{equation}\label{eq1547}
\rho \int_\Omega v_k^2(t) + \eta \int_0^t \int_\Omega | \nabla v_k |^2 + 2 \int_0^t \int_\Omega \sigma_k:Dv_k 
\leq \rho \int_\Omega v_k^2(0) + 2 c \norm{f}_{ \L{2}\V'}^2.
\end{equation}
Using the explicit expression of the stress tensor~$\sigma_k$ given by the graph definition~\eqref{eq:regular graph} and the fact that $q_k$ is non-negative, we have
\begin{equation}\label{positive rheo}
\sigma_k:Dv_k = q_k \frac{\abs{Dv_k}^2}{\abs{Dv_k}+1/k} \geq 0 .
\end{equation}
We obtain the estimate~\eqref{estimate2}, noting that the right-hand side member of~\eqref{eq1547} is bounded as follows
$$
 \rho \int_\Omega v_k^{2}(0) + c \norm{f}_{ \L{2}\V'}^2
  \leq
 \rho \| \mathcal P_k (v_0) \|_{\L{2}}^2 + c \norm{f}_{ \L{2}\V'}^2
  \leq
 \rho \| v_0 \|_{\L{2}}^2 + c \norm{f}_{ \L{2}\V'}^2
  \leq
 c.
$$
   \item In order to evaluate the term~$\partial_t v_k$ in $\V'$, we begin with any $\psi \in \V_k$ and rewrite~\eqref{model k}$_1$ as
	$$ \rho \left\langle \partial_t v_k, \psi \right\rangle_{\V',\V} = - \rho \int_\Omega (v_k \cdot \nabla v_k) \cdot \psi - \eta \int_\Omega \nabla v_k : \nabla \psi - \int_\Omega \sigma_k : \nabla \psi + \left\langle f, \psi \right\rangle_{\V',\V}. $$
Then, using the divergence-free constraint in the first term of the right hand side, the following inequality holds
	$$ \rho \left| \left\langle \partial_t v_k, \psi \right\rangle_{\V',\V} \right| \leq  \rho \int_\Omega \left| v_k \right| \, \left| \nabla \psi \right| \, |v_k| + \eta \int_\Omega \left| \nabla v_k \right| \, \left| \nabla \psi \right| + \int_\Omega |\sigma_k| \, \left|\nabla \psi\right| + \left| \left\langle f, \psi \right\rangle_{\V',\V}\right|. $$
Now applying H\"older and Poincar\'e inequalities, for all~$\psi$ in~$\V_k$, $\psi \neq 0$, we get
	$$ \rho \, \frac{ \left| \left\langle \partial_t v_k, \psi \right\rangle_{\V',\V} \right|}{\norm{\psi}_{\V}} 
	\leq  
  \rho \, \norm{v_k}_{\L 4} \, \norm{v_k}_{\L 4} + \eta \, \norm{v_k}_{\H 1} + \norm{\sigma_k}_{\L 2} + \norm{f}_{\V'}. $$
Since we work in $\Omega \subset \mathbb R^2$, the velocity~$v_k$ belongs to~$\in\H 1$. As a consequence of Poincar\'e inequality (see for example~\cite{B-F}) we have~$ \norm{v_k}_{\L 4}^2 \leq c \, \norm{v_k}_{\L 2} \norm{v_k}_{\H 1}$. 
By integrating on~$(0,T)$, we get
$$
\begin{aligned}
\rho \norm{ \partial_t v_k}_{\L{2}\mathrm{V}'} & \leq c \, \left( \int_0^T \big(\norm{v_k}_{\L{2}} \norm{ v_k}_{\H 1} + \norm{v_k}_{\H 1} + \norm{\sigma_k}_{\L{2}} + \norm{f}_{\V'} \big)^{2} dt \right) ^{ \frac{1}{2}} \\
                                               & \leq c \, \big( \norm{v_k}_{\L\infty\L{2}} \norm{ v_k}_{\L 2 \H 1} + \norm{v_k}_{\L 2 \H 1} + \norm{\sigma_k}_{\L 2 \L{2}} + \norm{f}_{\L 2 \V'}\big).
\end{aligned}
$$
We conclude thanks to~\eqref{estimate1} and~\eqref{estimate2}.
   \item The first estimate in~\eqref{estimate3} is obtained using $r=2p_{f,k}$ as test-function in the weak formulation~\eqref{model k}:
$$
\int_\Omega  p_{f,k} ^2 (t) + 2 k \int_0^t \int_\Omega | \nabla p_{f,k} |^2 = \int_\Omega  p_{f,k} ^2 (0) \leq c.
$$
In order to obtain the second one, we proceed similarly to obtain estimates on the velocity time derivative: we begin with any $r \in \H 1_{0,k}$, $r\neq 0$ in~\eqref{model k}$_2$ and we conclude thanks to the previous results~\eqref{estimate1},~\eqref{estimate2}~and~$(\ref{estimate3})_1$.
\end{enumerate}
\end{proof}

\subsection{Existence of limits}

\paragraph{Limit for the velocity}
From~\eqref{estimate2} and~\eqref{estimate4}, and using Aubin–Lions–Simon lemma~(\cf \cite{B-F}), we deduce that a velocity~$v$ exists such that, up to a sub-sequence, we have the following weak, weak~$\star$ and strong convergences:
\begin{enumerate}
	\item[\namedlabel{cv1}{(C1)}] $v_k \rightharpoonup v$ in $\L{2}\V$,
	\item[\namedlabel{cv1'}{(C1')}] $v_k \rightharpoonup^\star v$ in $\L{\infty}\H{}$,
	\item[\namedlabel{cv2}{(C2)}] $v_k \rightarrow v$ in $\L{2}\H{}$.
\end{enumerate}

\paragraph{Limit for the fluid pressure}
Thanks to estimate~\eqref{estimate3} of Proposition~$\ref{prop:estim k 1}$ and Aubin–Lions–Simon lemma again, we can also extract a sub-sequence from $(p_{f,k})_{k\geq 1}$, which converges to a certain function~$p_f$ as follows
\begin{itemize}
	\item[\namedlabel{cv3}{(C3)}] $p_{f,k} \rightharpoonup p_f$ in $\L{2}\H{1}$,
	\item[\namedlabel{cv3'}{(C3')}] $p_{f,k} \rightharpoonup^\star p_f$ in $\L{\infty}\L{\infty}$,
	\item[\namedlabel{cv4}{(C4)}] $p_{f,k} \rightarrow p_f$ in $\L{2}\L{2}$.
\end{itemize}

\paragraph{Limit related to the Bingham tensor}
As the domain~$\Omega$ is bounded, using estimates~\eqref{estimate1}, there exist~$q$ and~$\sigma$ such that, up to a sub-sequence again:
\begin{itemize}
	\item[\namedlabel{cv5}{(C5)}] $q_k \rightarrow q$ in $\L{2}\L{2}$,
	\item[\namedlabel{cv6}{(C6)}] $\sigma_k \rightharpoonup \sigma$ in $\L{2}\L{2}_{sym}$.
\end{itemize}

\noindent Note here that convergences~\ref{cv1'} and~\ref{cv3'} prove that the limits~$v$ and~$p_f$ respectively belongs to~$\L \infty \H{}$ and~$\L{\infty}\L{\infty}$. We do not use it again in the sequel.
\subsection{Passage to the limit when~$k$ tends to infinity}\label{section2-3}

We now prove that the limit $(v,p_f,\sigma)$ is a solution of the problem~\eqref{formulation faible1}.

\subsubsection{Limit in equations and initial conditions}

\paragraph{Limit in equations}
In order to deal with non-linear terms at the limit when~$k$ tends to infinity, it suffices to obtain the compactness on approximated solution by establishing an estimate on the time derivative of the velocity. 
As for homogeneous Navier-Stokes equations, see for instance~\cite{B-F} for details, we can estimate $\tau_h v_k - v_k $, where~$\tau$ is a shift operator in time in a suitable space.\\
Let $j > 0$ be an integer and~$\psi \in \V_j$ be a test-function. 
Then for all~$k \geq j$,
$$ 
\rho \bigg( \left\langle \partial_t v_k, \psi_j \right\rangle_{\V',\V} + \int_\Omega (v_k \cdot \nabla v_k) \cdot \psi_j \bigg) + \eta \int_\Omega \nabla v_k : \nabla \psi_j + \int_\Omega \sigma_k : \nabla \psi_j = \left\langle f , \psi_j \right\rangle_{\V',\V}  .
$$
The integer~$j$ being given, the aim is to let~$k$ tend to infinity. 
From \ref{cv1} and \ref{cv6}, we already know that
\begin{itemize}
\item[$\bullet$] $\nabla v_k\rightharpoonup \nabla v \in \L{2}\L{2}$,
\item[$\bullet$] $\sigma_k \rightharpoonup \sigma \in \L{2}\L{2}_{sym}$.
\end{itemize}
In order to treat the terms containing~$\partial_t v_k$ and~$ v_k \cdot \nabla v_k$,we first estimate~$( \partial_t v_k)_{k\geq 1}$. This quantity is bounded in~$\L{2}\mathrm{V}'$, then it has a weak limit, which must be~$ \partial_t v$ by continuity of the time derivative. Consequently, we have
\begin{itemize}
\item[$\bullet$] $\partial_t v_k \rightharpoonup \partial_t v \in \L{2}\mathrm{V'}$.
\end{itemize}
Furthermore, we can deduce (see~\cite{B-F}, p.357-358 for details) from convergences on~$v_k$ that
\begin{itemize}
\item[$\bullet$] $ v_k \cdot \nabla v_k \rightharpoonup v \cdot \nabla v \in \L{ \frac{4}{3}}\L{ \frac{6}{5}}$.
\end{itemize}
We now take the limit when~$k$ tends to infinity and we get the following equality on~$\sigma$ and~$v$ : 
$$ 
\rho \bigg( \left\langle \partial_t v , \psi_j \right\rangle_{\V',\V} + \int_\Omega (v \cdot \nabla v) \cdot \psi_j \bigg) + \eta \int_\Omega \nabla v : \nabla \psi_j + \int_\Omega \sigma : \nabla \psi_j = \left\langle f , \psi_j \right\rangle_{\V',\V} .
$$
As~$(w_i)_{i\geq1}$ is total, if~$ \psi \in \mathrm{V} $, then with~$ \psi_{j} = \mathcal{P}_{j} (\psi)$,~$ (\psi_j)_j$ converges to~$ \psi$ in~$ \mathrm{V}$. 
So we can take the limit in each term when~$j$ tends to infinity, and prove that~$(v,\sigma)$ satisfies the equation~$(\ref{formulation faible1})_1$ for all~$ \psi \in \mathrm{V}$.
In a more classical way, we can prove that the fluid pressure~$p_f$ satisfies~$(\ref{formulation faible1})_2$ for all~$r \in \H{1}_0(\Omega)$.

\paragraph{Initial conditions}
We apply twice the theorem II$.5.13$ p.101 in~\cite{B-F} to prove that the limits~$v$ and~$p_f$ are continuous on~$[0,T]$. 
More precisely:\\
\noindent $\checkmark$ for the velocity, thanks to~\eqref{estimate2} and~\eqref{estimate4} we deduce that~$v\in \mathcal C ^0 ([0,T],\H{})$, and 
$$
v\vert_{t=0} = v_0 \text{ a.e.} 
$$
\noindent $\checkmark$ For the fluid pressure, from~\eqref{estimate3} we deduce that~$p_f\in \mathcal C ^0 ([0,T],\L 2(\Omega))$, and
$$
p_f\vert_{t=0} = {p_f}_0 \text{ a.e.} 
$$

\subsubsection{Passage to the limit in the graph $\mathcal{G}_k$}\label{limit-graph}

The last step of the proof is to show that $(\sigma,Dv) \in \mathcal{G}[q]$. More precisely, we already know that~$D v$ belongs to~$\L \infty \H{-1}\cap\L 2 \L 2$ and that~$\sigma$ belongs to $\L 2 \L 2 _{sym}$. We will prove that (see the definition~\eqref{definition-Gq} of the graph~$\mathcal{G}[q]$):
\begin{equation}\label{bingham12}
\abs{\sigma} \leq q
\quad \text{and} \quad
\sigma : Dv = q \abs{Dv}.
\end{equation}
\paragraph{Step~1 --}
To prove the first relation of~\eqref{bingham12} we note that by definition of~$\mathcal G _k$, see \eqref{eq:regular graph}, for all~$k \geq 0$, we have $ \sigma _k = q_k \frac{Dv_k}{\abs{Dv_k} + 1/k}$, so that
$$ \abs{\sigma_k} \leq q_k \quad \text{ in }(0,T)\times\Omega.$$
We easily deduce that for all~$(t, t+h) \subset (0,T)$ and for all subset~$\omega \subset \Omega$ we have
 $$\int_t^{t+h} \int_\omega \abs{\sigma_k} \leq \int_t^{t+h} \int_\omega q_k.$$
But we know (see \ref{cv5} and \ref{cv6}) that~$q_k \rightarrow q \in \L{2}\L{2}$ and~$\sigma_k \rightharpoonup \sigma \in \L{2}\L{2}_{sym}$. As the domain $\Omega$ is bounded, we can take the limit when~$k$ tends to infinity:
$$
\forall (t , t+h) \subset (0,T) \quad  \forall \omega \subset \Omega \qquad \norm{\sigma}_{\L{1}((t,t+h)\times \omega)} \leq \norm{q}_{\L{1}((t,t+h)\times \omega)}.
$$
So we finally conclude that~$\abs{\sigma } \leq q$ a.e. in $(0,T)\times\Omega$.
\paragraph{Step~2 --}
Next steps consist of checking the second relation of~\eqref{bingham12}.
Note first that the difference between $q_k \abs{Dv_k}$ and $\sigma_k:Dv_k$ converges to~$0$ when~$k$ tends to infinity.
Indeed, let's compute the difference:
$$
q_k \abs{Dv_k} - \sigma_k:Dv_k = q_k \frac{1}{k} \frac{\abs{Dv_k}}{\abs{Dv_k} + 1/k},
$$
and thanks to~\eqref{estimate1} the right-hand side is essentially bounded by $ \frac{1}{k} c$, and thus converges to $0$ in $ \L \infty$.
Therefore, it suffices to prove that 
\begin{itemize}
	\item[(a)] $ q_k \abs{Dv_k} \rightharpoonup q \abs{Dv}$ in $\L 1 \L 1$
\end{itemize}
and
\begin{itemize}
	\item[(b)] $\sigma_k:Dv_k  \rightharpoonup \sigma:Dv $ in $\L 1 \L 1$.
\end{itemize}
\paragraph{Step~2~(a) --}
Due to strong-weak convergence results, as~$q_k \rightarrow q \in \L{2}\L{2}$ (see \ref{cv5}) and $Dv_k \rightharpoonup Dv \in \L{2}\L{2}$, see \ref{cv1}, we have
$$q_k \abs{Dv_k} \rightharpoonup q \abs{Dv} \in \L{1}\L{1}.$$
Now we must prove~(b).
\paragraph{Step~2~(b).1 --}
Using $\psi = v_k $ as a test-function in the first equation of~\eqref{model k} and taking into account divergence-free constraint, we obtain
\begin{equation}\label{eq:2.1.1}
\int_\Omega \sigma_k:Dv_k = \left\langle f , v_k \right\rangle_{\V',\V} - \frac1 2 \rho \, \d_t \int_\Omega |v_k |^2 - \eta \int_\Omega | \nabla v_k |^2.
\end{equation}
Integrating~\eqref{eq:2.1.1} with respect to the time between~$0$ and $t \in (0,T)$, we have
\[
\int_0^t \int_\Omega \sigma_k:Dv_k 
=
\int_0^t \left\langle f, v_k \right\rangle_{\V ',\V} + \frac1 2 \rho \, \norm{v_k(0)}_{\L{2}}^2 - \frac1 2 \rho \, \norm{v_k(t)}_{\L{2}}^2  -\eta \int_0^t \norm{\nabla v_k}_{\L{2}}^2.
\]
Taking the limit $k \to + \infty$ (and using the fact that $\displaystyle \liminf_{k \rightarrow \infty} \norm{g_k}_X \geq \norm{g}_X$ as soon as  $g_k \rightharpoonup g \in X$, see~\cite{Bre}), we deduce
\begin{equation}\label{eq:2.2}
\limsup_{k \rightarrow + \infty} \int_0^t \int_\Omega \sigma_k:Dv_k \leq 
\int_0^t \left\langle f, v \right\rangle_{\V ',\V} + \frac 1 2 \rho \, \norm{v_0}_{\L{2}}^2 - \frac 1 2 \rho \, \norm{v(t)}_{\L{2}}^2 - \eta \int_0^t \norm{\nabla v}_{\L{2}}^2 .
\end{equation}
If we now choose $\psi = v $ as a test-function in~\eqref{formulation faible1}$_1$ (see Remark~\ref{3D} for more precision) and use the divergence-free constraint, we obtain
\begin{equation}\label{eq:2.1.2}
\int_\Omega \sigma:Dv  = \left\langle f , v \right\rangle_{\V',\V} -  \frac1 2 \rho \, \d_t \int_\Omega |v|^2 - \eta \int_\Omega | \nabla v |^2.
\end{equation}

\noindent Then we integrate in time the equality~\eqref{eq:2.1.2} from~$0$ to~$t$:
\begin{equation}\label{eq:2.3}
\int_0^t \int_\Omega \sigma:Dv = \int_0^t \left\langle f, v \right\rangle_{\V ',\V} + \frac1 2 \rho\norm{v_0}_{\L{2}}^2 - \frac1 2 \rho\norm{v(t)}_{\L{2}}^2 - \eta \int_0^t \norm{\nabla v}_{\L{2}}^2 .
\end{equation}
Therefore, combining this equation~\eqref{eq:2.3} with the inequality~\eqref{eq:2.2} it follows
\begin{equation}\label{result intermed}
\limsup_{k \rightarrow + \infty} \int_0^t \int_\Omega \sigma_k:Dv_k \leq \int_0^t \int_\Omega \sigma:Dv .
\end{equation}

\paragraph{Step~2~(b).2 --} Now we use the fact that we have an idea of what~$\sigma$ must be. 
Indeed, thanks to~\eqref{bingham1}, when~$Dv$ does not vanish,~$\sigma$ is equal to $ q \frac{Dv}{\abs{Dv}}$, but when~$ Dv$ vanishes we don't know the value of~$ \sigma$. 
The idea is that we can choose zero as a default value. 
In this way, let's define two explicit functions~$\sigma ^* = \sigma ^* (q,A)$ and~$\sigma_k ^* = \sigma_k ^* (q,A)$ on $\mathbb{R} \times S_d(\mathbb{R})$ by
$$
\sigma^*(q,A)   = \left\lbrace \begin{array}{cc} q \frac{A}{\abs{A}} & \quad \mathrm{if} \; A \neq 0 \\ 0 & \mathrm{either,} \end{array}  \right. 
$$
and for any $k \in \mathbb N ^*$,
$$ 
\sigma_k^*(q,A) = \left\lbrace \begin{array}{cc} q \frac{A}{\abs{A}+1/k} & \quad \mathrm{if} \; A \neq 0 \\ 0 & \mathrm{either.} \end{array}  \right. 
$$
By the density of~$ \mathcal{D}((0,T)\times\Omega) $ in~$ \L{2}\L{2}$, we consider a sequence~$(D_i)_{i\geq1} \subset \mathcal{D}((0,T)\times\Omega)$, which satisfies
\begin{itemize}
	\item[\namedlabel{cv7}{(C7)}] $D_i \xrightarrow{i}{} Dv \in \L{2}\L{2}$.
\end{itemize} 
Noting that $\left( \sigma^*(q,D_i) \right) _{i\geq1}$ is bounded in~$ \L{2}\L{2} $, we can extract a sub-sequence (again indexed by $i\geq 1$), which converges in~$ \L{2}\L{2} $. 
Let's denote the limit by~$ \chi $, so that we have 
\begin{equation}
\sigma^*(q,D_i) \rightharpoonup \chi \, \in \L{2}\L{2}.
\label{sigma i-convergence}
\end{equation}
Furthermore, for any integer~$i$, the sequence~$\left( \sigma_k^*(q,D_i) \right) _{k\geq1}$ converges and the limit when~$k$ tends to infinity is~$ \sigma ^* (q,D_i) $.\\
\noindent
In order to prove a convergence in $\L 1 \L 1$, let's fix the index~$i$ and estimate the following quantity as~$k$ tends to infinity:
$$
\begin{aligned}
\int_0^t \int_\Omega (\sigma_k - \sigma_k^*(q_k,D_i)):(Dv_k - D_i) & = \int_0^t \int_\Omega \sigma_k: Dv_k - \int_0^t \int_\Omega \sigma_k :D_i \\
 & + \int_0^t \int_\Omega \sigma_k^*(q_k,D_i) : D_i - \int_0^t \int_\Omega \sigma_k^*(q_k,D_i) : Dv_k .
\end{aligned}
$$
We now look at the convergence for each term in the right-hand side except for the first one:
\begin{itemize}
\item[$\checkmark$] As $\sigma_k \rightharpoonup \sigma \in \L{2}\L{2}$ (see \ref{cv6}), we have 
$$\limsup_{k \rightarrow \infty} \left( - \int_0^t \int_\Omega \sigma_k : D_i \right) = - \int_0^t \int_\Omega \sigma : D_i .$$
\item[$\checkmark$] Thanks to the strong convergence of $(q_k)_{k\geq1}$ in $\L 2 \L 2$, we can say that~$\sigma_k^*(q_k,D_i) \xrightarrow[k]{} \sigma^*(q,D_i) \in \L{2}\L{2}$ and we deduce that 
$$\limsup_{k \rightarrow \infty} \left( \int_0^t \int_\Omega \sigma_k^*(q_k,D_i):D_i   \right) = \int_0^t \int_\Omega \sigma^*(q,D_i):D_i  .$$
\item[$\checkmark$] From convergence~\ref{cv1} and considering again that~$\sigma_k^*(q_k,D_i)\xrightarrow[k]{} \sigma^*(q,D_i) \in \L{2}\L{2} $, we apply the weak-strong convergence theorem to deduce that
$$\limsup_{k \rightarrow \infty} \left( \int_0^t \int_\Omega \sigma_k^*(q_k,D_i):Dv_k   \right) = \int_0^t \int_\Omega \sigma^*(q,D_i):Dv  .$$
\end{itemize}
Therefore, for any positive integer~$i$ we have
$$
\begin{aligned}
\limsup_{k \rightarrow \infty} \left( \int_0^t \int_\Omega (\sigma_k - \sigma_k^*(q_k,D_i)):(Dv_k - D_i)   \right) = & \limsup_{k \rightarrow \infty} \left( \int_0^t \int_\Omega \sigma_k: Dv_k   \right) - \int_0^t \int_\Omega \sigma : D_i  \\
 & + \int_0^t \int_\Omega \sigma^*(q,D_i) : (D_i - Dv)  .
\end{aligned}
$$
Now let~$i$ tend to infinity. 
Recalling that~$(D_i)_{i\geq1}$ converges strongly to~$Dv$ in~$\L{2}\L{2}$ and using~\eqref{sigma i-convergence}, the last term in the right-hand side tends to zero.
Finally, thanks to \eqref{result intermed}, we have
\begin{equation}\label{negativity}
\begin{aligned}
\lim_{i\rightarrow \infty}  \, & \limsup_{k \rightarrow \infty} \left( \int_0^t \int_\Omega (\sigma_k - \sigma_k^*(q_k,D_i)):(Dv_k - D_i)   \right) \\
 & \qquad = \limsup_{k \rightarrow \infty} \left( \int_0^t \int_\Omega \sigma_k: Dv_k \right) - \int_0^t \int_\Omega \sigma : Dv   \leq 0 .
\end{aligned}
\end{equation}

\paragraph{Step~2~(b).3 --} \label{pagelemme}
As  $(\sigma_k^*(q_k,D_i),D_i)$ and $(\sigma_k,Dv_k)$ are in the same graph~$\mathcal{G}_k$, thanks to Lemma~$\ref{graphe positif}$ we have 
$$ (\sigma_k - \sigma_k^*(q_k,D_i)):(Dv_k - D_i) \geq 0 .$$
Therefore, ``$\limsup_k$'' which appears above reads ``$\lim_k$" and inequalities in $ ( \ref{negativity})$ are equalities, \ie
$$
\lim_{i,k\rightarrow \infty} \int_0^t \int_\Omega \left| (\sigma_k - \sigma_k^*(q_k,D_i)):(Dv_k - D_i)   \right| = 0 ,\\
$$
\ie $(\sigma_k - \sigma_k^*(q_k,D_i)):(Dv_k - D_i)$ converges to $0$ in~$\L 1 \L 1$ when $k$ and $i$ tend to infinity, so that for any $\varphi \in \L\infty$, we have 
\begin{equation}
\lim_{i,k\rightarrow \infty} \int_0^t \int_\Omega \left( (\sigma_k - \sigma_k^*(q_k,D_i)):(Dv_k - D_i) \right) \varphi = 0. 
\label{eq:limite L1}
\end{equation}
We now look for the limit of~$\sigma_k^*(q_k,D_i)$. 
Note that for all~$i \geq 1$ and~$k \geq 1$, $\norm{\sigma_k^*(q_k,D_i)}_{\L 2 \L 2} \leq \norm{q}_{\L 2 \L 2}$, then we can chose any index between~$i$ and~$k$ first to compute the limit. 
For more convenience, we begin with~$i$. 
Therefore, we set any~$k \geq 1$ and compute the difference 
$$ 
\begin{aligned}
\sigma_k^*(q_k,D_i) - \sigma_k^*(q_k,Dv) & = q_k\frac{D_i}{\abs{D_i}+1/k} - q_k\frac{Dv}{\abs{D_i}+1/k} 
																						+ q_k\frac{Dv}{\abs{D_i}+1/k} - q_k\frac{Dv}{\abs{Dv}+1/k} \\
																				 & = q_k \left( \frac{D_i-Dv}{\abs{D_i}+1/k} + Dv \frac{ \abs{Dv}-\abs{D_i} }{ (\abs{D_i}+1/k)(\abs{Dv}+1/k) } \right).
\end{aligned}
$$
Thereby, thanks to~\ref{cv7} we have 
$$ \sigma_k^*(q_k,D_i) \xrightarrow{i}{} \sigma_k^*(q_k,Dv) \in \L 2 \L 2 .$$
Then, we can observe that 
$$ \frac{Dv}{\abs{Dv}+1/k} \text{ converges almost surely in } (0,T)\times \Omega \text{ to } \left\{ \begin{array}{ll} \frac{Dv}{\abs{Dv}} & \text{if } Dv \neq 0 \\ 0 & \text{if } Dv = 0 . \end{array} \right. $$
As the sequence $ \left(\frac{Dv}{\abs{Dv}+1/k}\right)_{k\geq1} $ is uniformly bounded by $1$, now let~$k$ tend to infinity and apply Lebesgue's dominated convergence theorem, so that 
$$
\frac{Dv}{\abs{Dv}+1/k} \xrightarrow{k}{} \left\{ \begin{array}{ll} \frac{Dv}{\abs{Dv}} & \text{if } Dv \neq 0 \\ ~~ 0 & \text{if } Dv = 0 \end{array} \right. \in \L 2 \L 2.
$$
Then, thanks to~\ref{cv5} the product~$q_k \frac{Dv}{\abs{Dv} + 1/k} = \sigma_k^*(q_k,Dv)$ converges strongly in~$\L 1 \L 1$:
$$ \sigma_k^*(q_k,Dv) \xrightarrow{k}{} \sigma^*(q,Dv) \in \L 1 \L 1 .$$
But we know that $\sigma_k^*(q_k,Dv) - \sigma^*(q,Dv) $ is uniformly essentially bounded in~$(0,T)\times\Omega$ (see \eqref{estimate1}$_2$).
Hence the difference converging to~$0$ also in $\L 2 \L 2$, and we have
\begin{equation}\label{i-k-convergence}
\sigma_k^*(q_k,D_i) \xrightarrow{i,k}{} \sigma^*(q,Dv) \in \L 2 \L 2 . 
\end{equation}
Combining~\ref{cv1} with~\eqref{i-k-convergence}, we automatically get 
$$ \sigma_k^*(q_k,D_i):(Dv_k - D_i) \underset{i,k}{\rightharpoonup} 0 \in \L 1 \L 1.$$
Finally from \eqref{eq:limite L1} and for any $\varphi \in \L\infty$, the following convergence remains:
$$
\lim_{i,k\rightarrow \infty} \int_0^t \int_\Omega \left( \sigma_k:(Dv_k - D_i) \right) \varphi = 0,
$$
\ie 
$$
\lim_{k\rightarrow \infty} \int_0^t \int_\Omega (\sigma_k:Dv_k) \varphi = \lim_{i,k\rightarrow \infty} \int_0^t \int_\Omega (\sigma_k: D_i) \varphi.
$$
But from \ref{cv6} and \ref{cv7} we compute the limit on the right-hand side above, which is equal to $\int_0^t \int_\Omega (\sigma: Dv) \varphi$.
Thereby, we prove that
$$ \sigma_k : Dv_k \rightharpoonup \sigma : Dv \in \L{1}\L{1}.$$
From result of Step~$2$ (a) and (b), by uniqueness of the limit we have the equality 
$$\sigma:Dv = q \abs{Dv} \text{ a.e.}$$
This proves the second relation of~\eqref{bingham12} so $(\sigma, Dv) \in \mathcal{G}[q]$, which concludes the proof.

\begin{rem}\label{3D}
As for test-function $\psi = v_k$, it is theoretically not possible to choose $\psi = v$ as a test-function in~\eqref{formulation faible1}. 
But in the case of bi-dimensional flows, we have 
$$ \partial_t v \in \L{2}\V'. $$
So if we regularize the velocity~$v$ with respect to the time, say~$v_\varepsilon = \eta_\varepsilon(t) \star v$ and~$\lim_{\varepsilon \rightarrow 0} v_\varepsilon = v$ in~$\L{2}\V$, 
then we can use it instead of~$v$ and take the limit as~$\varepsilon$ tends to~$0$. The result is the one we obtain by formally using~$v$ as a test-function.\\
Note that this method does not work in a three-dimensional case, because we only have
$$ \partial_t v \in \L{\frac{4}{3}}\V'. $$
\end{rem}

\section{An example of flow with variable plasticity yield}\label{section4}

In this section, we provide a numerical illustration of the previous model which highlights the effect of fluidization on the flow. We consider the following model:
\begin{equation}\label{modelsimpli00}
\left\{
\begin{aligned}
& \partial_t v + v \cdot \nabla v + \nabla p - \Delta v = 0.2 \, \div \Big( (1 - p_f)^+ \frac{Dv}{\abs{Dv}} \Big), \\
& \div v = 0, \\[0.2cm]
& \partial_t p_f + v \cdot \nabla p_f - 0.1 \, \Delta p_f = 0.
\end{aligned}
\right.
\end{equation}
In order to illustrate the influence of the diffusion of the fluid pressure~$p_f$, we look at the behaviour of a fluid with such a rheology in a $2$D-channel.
The fluid enters the channel with a Poiseuille-type profile for velocity and flows out freely.
Its rheology depends on the degree of fluidization, which is just given by~\eqref{def q}. 
When testing the behaviour of a simple Bingham fluid in this configuration, we set~$p_f=0$ everywhere, so that $(1-p_f)^+ = 0$.
But when we compute the same flow with a defluidization process, the fluid pressure equals the lithostatic pressure everywhere inside at the beginning and at the inlet, so that the fluid is initially totally fluidized and when entering the channel.
The corresponding dimensionless boundary conditions for the interstitial pore pressure~$p_f$ are given by~$1$ at the inlet and~$0$ (\ie atmospheric pressure) on the other boundaries.
Initial velocity is given by the steady state solution of the equivalent Poiseuille-type problem for a pure viscous fluid and initial fluid pressure depends on the experimental set-up.\\

\noindent
The numerical scheme used to compute the solution we present here is described and analyzed in~\cite{Ch-Du}.
It is a second-order projection method with respect to the time based on the combination of the BDF2 (second-order Backward Differentiation Formula) and the AB2 (Adams--Bashforth). 
The non-linearity due to the Bingham threshold is treated using a projection method (see~\cite{Ch-Du} for details on the numerical analysis). 
A Cartesian uniform mesh is used for the spatial discretisation, and a classical MAC\footnote{Marker And Cell, see~\cite{MAC}} scheme is introduced in order to treat the incompressibility.\\

\noindent
The parameters of the momentum equation for these numerical tests are usual ones (see~\cite{Mura}).
Other concrete examples of this type of model applied to real physical problems represent a work in progress.
We illustrate the influence of the presence of the fluid pressure through the geometry of the rigid zones. 
In the case of a pure Bingham fluid, in the steady state the rigid zone corresponds to the plug in the middle of the channel.
We can observe that in the Figure~\ref{fig:Bingham}, describing the flow corresponding to an established Bingham flow.
We can observe that the Poiseuille flow imposed on the left rapidly becomes a plug flow characterizing the Bingham flow.
\begin{figure}[h!]%
\includegraphics[width=\columnwidth]{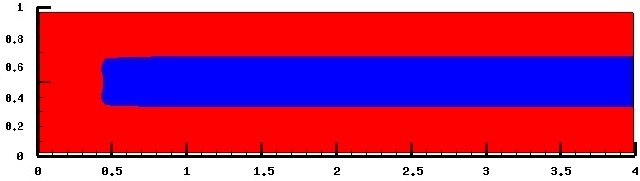}\\ 
\includegraphics[width=\columnwidth]{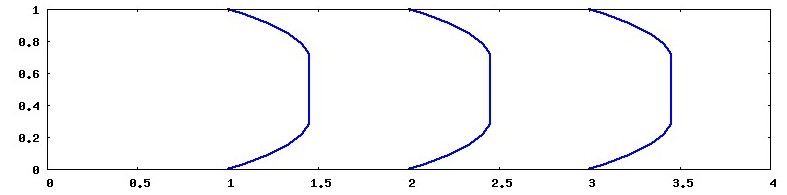}   
\caption{The non-fluidized case (corresponding to the Bingham model) at $t = 30$. 
Top: Rigid zone (in the middle of the channel) developed in the flow. 
Bottom: Horizontal velocity profile at different positions.}%
\label{fig:Bingham}%
\end{figure}

\noindent
In a second simulation, we are looking at in the same quantities (plug zones and velocity profile) for a fluid taking into account the gas pressure.
The rigid zone is thinner where the parameter~$q$ does not vanish, whereas it has quite the same behaviour where~$q = 0$ (see Figure~\ref{fig:fluidised} and~\ref{fig:q}). 
Clearly the plug flow is being established further from the source corresponding to $x=0$.
\begin{figure}[h!]%
\includegraphics[width=\columnwidth]{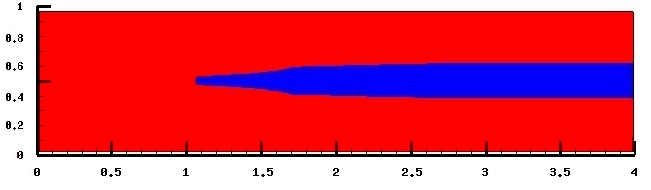}\\ 
\includegraphics[width=\columnwidth]{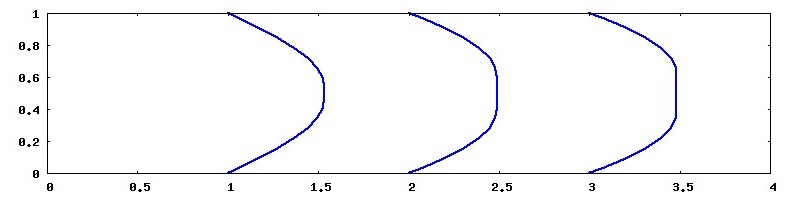}   
\caption{The fluidized case (variable plasticity yield) at $t = 30$. 
Top: Rigid zone (in the middle of the channel) developed in the flow. 
Bottom: Horizontal velocity profile at different positions.}%
\label{fig:fluidised}%
\end{figure}

\begin{figure}[h!]%
\includegraphics[width=\columnwidth]{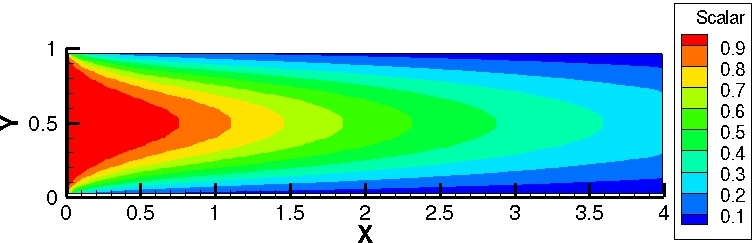}   
\caption{Distribution of the interstitial gas pressure~$p_f$ in the flow of a Bingham fluid at $t=30$.}%
\label{fig:q}%
\end{figure}

\newpage
\section{Conclusions and prospects}\label{section5}

We have established existence of a weak solution to bi-dimensional homogeneous Bingham fluids with variable shear stress, the main new feature of the model being the coupling of the rheology with another partial differential equation through the definition of the order-parameter.
This work is an extension in the field of existence theorems for homogeneous fluids (Navier-Stokes \cite{B-F}, Bingham, Herschel-Bulkley \cite{B-G-M-S}).
In order to compute some free surface-like flows, it would be interesting to complete the study by taking into account a variable density. 
Some results already exist for Bingham~\cite{A-S} and Herschel-Bulkley~\cite{bas} fluids, but only with constant plasticity yield.
Also, following the case treated in~\cite{B-G-M-S}, it seems to be possible to extend the result to the three-dimensional case and thus to include more physical flows.

\section*{Acknowledgments}
\emph{We would like to thank Karim Kelfoun and Olivier Roche (L.M.V.) for their active participation to the development of the ideas of this article, especially in the creation of the model.
This research was partially supported by the French Government Laboratory of Excellence initiative noANR-10-LABX-0006, the R\'egion Auvergne and the European Regional Development Fund. This is Laboratory of Excellence ClerVolc contribution number 218.}

\end{document}